\documentclass{amsart}
\usepackage{amsmath, amsthm, amssymb, amsfonts}
\usepackage[normalem]{ulem}
\usepackage{hyperref}

\usepackage{verbatim}
\usepackage{caption}
\usepackage{mathtools}

\usepackage{tikz-cd}

\theoremstyle{plain}
\newtheorem{thm}{Theorem}[section]
\newtheorem{cor}[thm]{Corollary}

\newtheorem{lemma}[thm]{Lemma}

\theoremstyle{definition}

\theoremstyle{remark}
\newtheorem{rmk}[thm]{Remark}

\newcommand{\BQ}{{\mathbb{Q}}}
\newcommand{\BR}{{\mathbb{R}}}

\newcommand{\BZ}{{\mathbb{Z}}}

\newcommand{\CE}{{\mathcal E}}

\newcommand{\CI}{{\mathcal I}}

\newcommand{\CL}{{\mathcal L}}

\newcommand{\Fg}{{\mathfrak{g}}}

\newcommand{\Fq}{{\mathfrak{q}}}

\DeclareMathOperator{\Hilb}{Hilb}

\DeclareFontFamily{OT1}{rsfs}{}
\DeclareFontShape{OT1}{rsfs}{n}{it}{<-> rsfs10}{}
\DeclareMathAlphabet{\curly}{OT1}{rsfs}{n}{it}

\newcommand\ext{\curly Ext}

\newcommand{\p}{\mathbb{P}}

\newcommand\id{\operatorname{id}}

\newcommand\End{\operatorname{End}}

\newcommand{\Pic}{\mathop{\rm Pic}\nolimits}

\renewcommand{\sl}{{\mathfrak{s}\mathfrak{l}}}
\newcommand{\so}{{\mathfrak{s}\mathfrak{o}}}
\newcommand{\SO}{{\mathrm{SO}}}

\begin{document}
\title[A Lie algebra action]{
A Lie algebra action on the Chow ring of the Hilbert scheme of points of a K3 surface}
\date{\today}

\author{Georg Oberdieck}
\address{University of Bonn, Institut f\"ur Mathematik}
\email{georgo@math.uni-bonn.de}

\begin{abstract}
We construct an action of the Neron--Severi part of the Looijenga-Lunts-Verbitsky Lie algebra
on the Chow ring of the Hilbert scheme of points on a K3 surface.
This yields a simplification of Maulik and Negut's proof that the cycle class map is injective on the subring generated by divisor classes
as conjectured by Beauville.
The key step in the construction is an explicit formula for Lefschetz duals in terms of Nakajima operators. 
Our results also lead to a formula for the monodromy action on Hilbert schemes in terms of Nakajima operators.
\end{abstract}

\maketitle

\section{Introduction}
\subsection{Chow}
Let $X$ be a smooth complex projective variety of dimension $m$.
Let $h \in \End H^{\ast}(X,\BQ)$ be the operator that acts on $H^i(X,\BQ)$ by multiplication with $i-m$.
Let also $e_a \in \End H^{\ast}(X,\BQ)$ denote the operator of cup product with a given element $a \in H^2(X,\BQ)$.
The element $a$ is called \emph{Lefschetz} if there exists an operator $f_a \in \End H^{\ast}(X,\BQ)$
such that $e_a, f_a, h$ satisfy the $\sl_2$-commutation relations
\[ [e_a, f_a] = h, \quad [h,e_a] = 2 e_a, \quad [h, f_a] = -2 f_a. \]
In this case we say $(e_a, f_a, h)$ is a Lefschetz triple.
The operator $f_a$, if it exists, is unique and is called the Lefschetz dual to $e_a$. 
By the Hard Lefschetz theorem every ample class on $X$ is Lefschetz.
More generally, an element $a$ is Lefschetz precisely if the morphism
$e_a^{s} : H^{m-s}(X) \to H^{m + s}(X)$
is an isomorphism for every $s \geq 0$.
In particular being Lefschetz is a Zariski open condition.

The total Lie algebra of $X$ introduced by Looijenga and Lunts \cite{LL}
and Verbitsky \cite{V} is the Lie subalgebra
\[ \Fg(X) \subset \End H^{\ast}(X,\BQ) \]
generated by all Lefschetz triples $(e_a, f_a, h)$.
We also consider the Neron-Severi Lie algebra of $X$ which is defined as the Lie subalgebra
\[ \Fg_{\mathrm{NS}}(X) \subset \Fg(X) \]
generated by all Lefschetz triples such that $a$ is algebraic, i.e. $a \in H^{1,1}(X,\BQ)$.

Assume now that $X$ is irreducible holomorphic symplectic, that is it is simply connected
and $H^0(X, \Omega_X^2)$ is generated by a holomorphic symplectic form $\sigma$.
The prime example of such a variety is the Hilbert scheme of points of a K3 surface.
By a result of Verbitsky \cite{V} and Looijenga and Lunts \cite{LL} we have
\[ \Fg(X) \otimes \BR = \so_{\BR}(4,b_2(X)-2), \quad \Fg_{\mathrm{NS}}(X) \otimes \BR = \so_{\BR}(2, \rho(X)) \]
where $b_i(X)$ are the Betti numbers and $\rho(X)$ is the Picard rank of $X$.

Let $A^{\ast}(X)$ denote the Chow ring of $X$ taken here always with $\BQ$-coefficients.
The group of correspondences $A^{\ast}(X \times X)$ carries a natural ring structure given by composition.
The cycle class map $\mathrm{cl} : A^{\ast}(X \times X) \to \End H^{\ast}(X)$ is a ring homomorphism.
Our main result says that 
for Hilbert schemes of points of K3 surfaces
the action of the Neron-Severi Lie algebra on cohomology lifts to an action on Chow groups by correspondences:

\vspace{5pt}
\begin{thm} \label{Thm1}
Let $X$ be the Hilbert scheme of points of a smooth projective K3 surface.
There exists a Lie algebra homomorphism $\rho : \Fg_{\mathrm{NS}}(X) \to A^{\ast}(X \times X)$ such that the following diagram commutes:
\[
\begin{tikzcd}
\Fg_{\mathrm{NS}}(X) \ar{r}{\rho} \ar{dr} & A^{\ast}(X \times X) \ar{d}{\mathrm{cl}} \\
& \End H^{\ast}(X,\BQ).
\end{tikzcd}
\]
\end{thm}
\vspace{7pt}

We will prove Theorem~\ref{Thm1} in Section~\ref{Section_Formulas}  by explicitly constructing a representation $\rho$ with the desired properties.
Under our $\rho$ the operators $e_a$ lift to the cup product with the divisor class $a$. 
The claim that the Lefschetz dual $f_a$ lifts to Chow
is precisely the Grothendieck standard conjecture of Lefschetz type \cite{Kl}.
The standard conjectures have been proven for Hilbert schemes of points on surfaces by Arapura \cite{Ar}.
The main improvement of our approach is that we give an explicit lift of the operator $f_a$
and show that also all the relations between the $e_a$'s and $f_b$'s lift.
For example, the operator $h$ lifts to an endomorphism that is diagonalizable on the group of zero cycles and whose eigenspaces
decomposition recovers the Beauville-Voisin decomposition, see Remark~\ref{Remark_h}.
A more general decomposition of the Chow motive of $X$ into eigenmotives under the action of $h$ is obtained in the subsequent work \cite{NOY}.

It is natural to expect that the conclusion of Theorem~\ref{Thm1} holds for all irreducible holomorphic symplectic varieties.
The standard conjectures for irreducible holomorphic symplectic varieties deformation equivalent to the Hilbert schemes of points of K3 surfaces have been proven by
Charles and Markman \cite{CM} by deforming (the lift of) the operator $f_a$.
The difficulty in extending Theorem~\ref{Thm1} beyond Hilbert schemes is to deform also the relations between the operators $e_a$ and $f_b$.

For the proof of the theorem we consider the action of the Nakajima operators $\Fq_n$ on the direct sum of Chow groups
\[ A^{\ast}( \Hilb(S)) = 
\bigoplus_{n=0}^{\infty} A^{\ast}( \Hilb^n(S) ), \]
where we let $\Hilb^n(S)$ denote the Hilbert scheme of $n$ points of a projective K3 surface $S$.
See also Section~\ref{Section_Nakajima_operators} for further details on Nakajima operators.
The action of the operators $e_a$ on cohomology was expressed in terms of Nakajima operators by Lehn \cite{Lehn}.
By recent work of Maulik and Negut \cite{MN} the formula of Lehn holds also in Chow.
We prove Theorem~\ref{Thm1} by explicitly writing lifts of the $f_a$ in terms of Nakajima operators and show
they satisfy the required commutation relations.

We have the following consequence of Theorem~\ref{Thm1} which was conjectured by Beauville and first proven by Maulik and Negut.
\begin{cor}[\cite{MN}] \label{cor} Let $S$ be a K3 surface.
The cycle map $A^{\ast}( \Hilb^n(S) ) \to H^{\ast}( \Hilb^n(S) )$ is injective on the subring
generated by divisor classes.
\end{cor}
\begin{proof}
The subring of $A^{\ast}(\Hilb^n( S))$ generated by divisor classes is an irreducible representation of
the simple Lie algebra $\Fg_{\mathrm{NS}}(X)$, hence the cycle class map restricted to it is either injective or zero.
\end{proof}
Theorem~\ref{Thm1} and hence also the proof of Corollary~\ref{cor} are not independent from \cite{MN}
and should be rather viewed as replacing its representation-theoretic part. 
In 
\cite{MN} 
Lehn's formula is used to construct an action of the product of the Heisenberg and Virasoro algebra on $A^{\ast}(\Hilb(S))$.
Beauville's conjecture is then deduced from Schur's Lemma.
Our approach here also relies on Schur's Lemma but has the advantange that the Lie algebra $\Fg_{\text{NS}}(X)$ involved is much smaller (for once it is finite-dimensional)
and that the argument might generalize to other cases.\footnote{On ther other hand,
Maulik and Negut's argument yields the stronger statement that the cycle class is injective on the subring generated by all small tautological classes \cite{MN}.}

\subsection{Application to monodromy}
Recall that the locus of Hilbert schemes of points of K3 surfaces is of codimension $1$ in the moduli space of irreducible holomorphic symplectic varieties.
In particular, for all $n \geq 2$ the monodromy group of $X = \Hilb^n(S)$ is strictly larger than the monodromy group of the underlying K3 surface $S$.
On the other hand the Nakajima operators define a basis of the cohomology of $\Hilb^n(S)$
which strongly depends on the Hilbert scheme structure.
A basic question is how the monodromy group acts on this basis,
and whether its action on cohomology can be written in terms of Nakajima operators.

In Theorem~\ref{Thm2} we describe the action of the total Lie algebra $\Fg(X)$ on cohomology in terms of Nakajima operators.
This leads to a formula for the monodromy action as follows.
The degree zero part of the Lie algebra $\Fg_0(X)$ is isomorphic to $\so( H^2(X,\BQ) ) \oplus \BQ h$
where $H^2(X,\BQ)$ is endowed with the Beauville-Bogomolov quadratic form.
Its action on $H^{\ast}(X)$ integrates to an action $\sigma : \SO( H^2(X,\BZ)) \to \End H^{\ast}(X,\BQ)$.
By a result of Markman \cite{Markman2} the monodromy group of $X$ is
\[ \mathrm{Mon}( X ) = \widetilde{O}^+( H^2( X,\BZ) ) \]
where the right hand side stands for orthogonal transformations which preserve the orientation
and act by $\pm 1$ on the discriminant.
By \cite[Lemma~4.13]{Markman} the monodromy action on cohomology agrees with $\sigma$
on an index $2$ subgroup of the intersection
\[ \SO( H^2(X,\BZ) ) \cap \widetilde{O}^+( H^2( X,\BZ) ). \]
This leads to the desired formulas up to finite index.

The description of the monodromy in terms of Nakajima operators was the original motivation for considering the operators $f_a$ in the Nakajima basis.
It will also play an important role in holomorphic anomaly equations for Hilbert schemes of points of K3 surfaces in forthcoming work.

\subsection{Plan}
In Section~\ref{Section_Prelim} we give preliminaries on the Lie algebra, Nakajima operators, and the Chow ring of K3 surfaces.
In Section~\ref{Section_Formulas} we state the formulas for the Lefschetz duals $f_a$ and give the proof of Theorem~\ref{Thm1}.

\subsection{Acknowledgements}
I would like thank Hsueh-Yung Lin and Andrei Negut for interesting discussions
on how one may deform the algebra action, and the latter for giving an inspiring talk in Bonn in June 2019.
I'm also very grateful to Junliang Shen and Qizheng Yin for useful discussions.

\section{Preliminaries} \label{Section_Prelim}
\subsection{The Lie algebra made explicit}
Let $V$ be a vector space with a non-degenerate symmetric bilinear form $( - , - )$ on it.
The wedge prouct $\wedge^2 V$ carries naturally the structure of a Lie algebra. The Lie bracket is defined by
\[ [a \wedge b, c \wedge d] = (a,d) b \wedge c - (a,c) b \wedge d - (b,d) a \wedge c + (b,c) a \wedge d \]
for all $a,b,c,d \in V$.
There exist a natural Lie algebra isomorphism $\wedge^2 V \to \so(V)$
by letting $a \wedge b$ act on $V$ via the endomorphism $(a \wedge b) v = (b,v) a - (a,v)b$.

Let $X$ be an irreducible holomorphic symplectic variety.
The Beauville-Bogomolov form is a non-degenerate quadratic form on $H^2(X,\BZ)$.
Let $U$ be the hyperbolic lattice $\binom{0\ 1}{1\ 0}$ with standard basis $e,f$.
By \cite{LL,V} we have
\[ \Fg(X) = \so( H^2(X,\BQ) \oplus U_{\BQ} ). \]
After identifying the right hand side with the second wedge product of $H^2(X,\BQ) \oplus U_{\BQ}$
as before, this isomorphism is given explicitly by
\[ e_a = e \wedge a, \quad f_a = \frac{-2}{(a,a)} f \wedge a, \quad h = 2 \cdot e \wedge f \]
for all $a \in H^2(X)$ with $(a,a) \neq 0$.
We will also use
\[ \widetilde{f}_a = -2 \cdot f \wedge a \]
which is defined for all $a$, is linear in $a$ and satisfies $\widetilde{f}_a = (a,a) f_a$ whenever $(a,a) \neq 0$.


\subsection{The Chow ring of a K3 surface}
Let $S$ be a smooth projective K3 surface and let
\[ c \in A^2(S) \]
be the class of any point on any rational curve of $S$.
Beauville and Voisin \cite{BV} prove the following basic relations:
\[ c_2(T_S) = 24 c, \quad \ell \cdot \ell' = ( \ell, \ell' ) c \]
for all $\ell, \ell' \in A^1(S)$.
They also establish the following decomposition of the class of the small diagonal $\Delta_{123}$ in the Chow ring of $S \times S \times S$:
\begin{equation} \label{decomp_delta123}
[ \Delta_{123} ] = \Delta_{12} c_3 + \Delta_{13} c_2 + \Delta_{23} c_1 - c_1 c_2 - c_1 c_3 - c_2 c_3,
\end{equation}
where following \cite{MN} we write $c_i$ for the pullback of $c$ along the projection to the $i$-th factor
and $\Delta_{ij}$ for the pullback of the class of the diagonal in $S^2$ along the projection to the $(i,j)$-factor, etc.
Parallel conventions will be followed throughout.
We will also use the following relation from \cite{BV}:
\begin{align*}
\Delta \cdot c_1 & = \Delta \cdot c_2 = c_1 \cdot c_2 \\
\quad \Delta \cdot \ell_1 & = \Delta \cdot \ell_2 = c_1 \ell_2 + \ell_1 c_2.
\end{align*} 

\subsection{Nakajima operators} \label{Section_Nakajima_operators}
Let $S$ be a smooth projective surface.
We recall the definition of Nakajima correspondences \cite{Nak, Groj} following the presentation of \cite{MN}.
Throughout the paper we will use the term `operator' synonymously with 'correspondence'.
Given a correspondence $\Gamma$ we will write $\Gamma$ also for the induced morphism on Chow groups.
Also note that the results of \cite{MN} were formulated in terms of homomorphisms of Chow groups,
but since the proofs are purely cycle-theoretic, they imply also the (in general stronger) parallel statements for correspondences.

For $n \geq 0$ and $i > 0$ consider the closed subscheme
\[ Z_{n,n+i} = \{ (\xi, x , \eta) \in \Hilb^n(S) \times S \times \Hilb^{n+i}(S) | \xi \subset \eta, \mathrm{Supp}(I_{\xi}/I_{\eta}) = \{ x \} \} \]
and let $p_1 : Z_{n,n+i} \to \Hilb^n(S)$, $p_2 : Z_{n,n+i} \to S$ and $p_3 : Z_{n,n+i} \to \Hilb^{n+i}(S)$ be the projection
to the factors.
The Nakajima operators are defined by 
\begin{align*}
\Fq_i  & = (p_2 \times p_3)_{\ast} p_1^{\ast} \\
\Fq_{-i} & = (-1)^i \cdot (p_1 \times p_2)_{\ast} p_3^{\ast}.
\end{align*}
We also set $\Fq_0 = 0$. Following \cite{MN} the $\Fq_i$ here are viewed as defining operators
\[ \Fq_i : A^{\ast}(\Hilb^n(S)) \to A^{\ast}(\Hilb^{n+i}(S) \times S). \]
The composition $\Fq_{i_1} \cdots \Fq_{i_k}$ of Nakajima operators is understood as an operator
\[ \Fq_{i_1} \cdots \Fq_{i_k} : A^{\ast}(\Hilb^n(S)) \to A^{\ast}(\Hilb^{n+i_1 + \ldots + i_k}(S) \times S^k) \]
where the operator $\Fq_{i_j}$ acts by its definition on the Hilbert scheme and by the identity on all remaining $S$-factors.
We have the Heisenberg commutation relations
\begin{equation} [\Fq_m, \Fq_n] = m \delta_{m+n,0} \mathrm{id} \times \Delta. \label{Nak_com} \end{equation}

For $\alpha \in A^{\ast}(S)$ we also write
\[ \Fq_i(\alpha) = p_{3 \ast}( p_1^{\ast}(\, \cdot \, ) \cup p_2^{\ast}(\alpha) ) \]
and similarly for negative $i$. The commutation relations read
\[ [\Fq_m(\alpha), \Fq_n(\beta)] = m \delta_{m+n,0} \langle \alpha, \beta \rangle \mathrm{id}. \]
More general given $\Gamma \in A^{\ast}(S^k)$ we let
\[ \Fq_{i_1} \cdots \Fq_{i_k}(\Gamma) : A^{\ast}(\Hilb^n(S)) \to A^{\ast}( \Hilb^{n + i_1 + \ldots + i_k}(S) ) \]
be the operator obtained by viewing $\Fq_{i_1} \cdots \Fq_{i_k}$
as a correspondence from $S^k$ to $\Hilb^{n + \sum_j i_j}(S)$ and applying it to $\Gamma$.

\section{Formulas and proofs} \label{Section_Formulas}
\subsection{Formulas}
Let $S$ be a smooth projective K3 surface.
Let $\Delta_{\Hilb^n(S)} \subset \Hilb^n(S)$ be the divisor parametrizing non-reduced subschemes
and let
\[ \delta = -\frac{1}{2} [ \Delta_{\Hilb^n(S)} ]. \]
By definition $\delta = 0$ if $n \leq 1$.
For all $n \geq 1$ we have the orthogonal decomposition
\begin{equation*} H^2(\Hilb^n(S), \BZ) \cong H^2(S, \BZ) \oplus \BZ \delta. \label{rewr} \end{equation*}
The restriction of the Beauville-Bogomolov form to the first factor is the intersection pairing on $S$.
Moreover, $(\delta, \delta) = 2-2n$.
Similarly, for algebraic classes we have
\begin{equation} A^1( \Hilb^n(S) ) \cong A^1(S) \oplus \BZ\delta. \label{a1iso} \end{equation}
We will identify classes in $A^1(S) \oplus \BZ\delta$ with their image in $A^1(\Hilb^n(S))$
under the isomorphism \eqref{a1iso} and similarly for cohomology.

Let $e_a$ be the operator
which acts on $A^{\ast}(\Hilb^n(S))$ by cup product with the class $a \in A^1(S) \oplus \BZ \delta$.
By the results of Lehn \cite[Thms. 3.5, 3.10]{Lehn} and Maulik-Negut \cite[Thm.1.6]{MN} we have 
for all $\alpha \in A^1(S)$ the equality
\begin{equation} \label{e_a}
\begin{gathered}
e_{\alpha} = -\sum_{n > 0} \Fq_{n} \Fq_{-n} ( \Delta_{\ast} \alpha) \\
e_{\delta} = -\frac{1}{6} \sum_{i+j+k=0} : \Fq_i \Fq_j \Fq_k ( \Delta_{123} ):
\end{gathered}
\end{equation} 
where $\Delta : S \to S^2$ is the inclusion of the diagonal,
and $: - :$ is the normal ordered product defined by
\[ : \Fq_{i_1} \cdots \Fq_{i_k} : \ =\  \Fq_{i_{\sigma(1)}} \cdots \Fq_{i_{\sigma(k)}} \]
where $\sigma$ is a permutation such that $i_{\sigma(1)} \geq \ldots \geq i_{\sigma(k)}$.

The formulas \eqref{e_a} hold also in cohomology for all $\alpha \in H^2(S,\BQ)$. 

Define the following operators on $A^{\ast}(\Hilb^n(S))$:
\begin{gather}
h = 2 \sum_{n > 0} \frac{1}{n} \Fq_{n} \Fq_{-n}( c_2 - c_1 ) \label{def_h} \\
\tilde{f}_{\alpha} = -2 \sum_{n > 0} \frac{1}{n^2} \Fq_{n} \Fq_{-n}( \alpha_1 + \alpha_2 ) \label{ft_a}\\
\widetilde{f}_{\delta}
=
-\frac{1}{3} \sum_{i+j+k=0} :\Fq_i \Fq_j \Fq_k \left( \frac{1}{k^2}  \Delta_{12}  + \frac{1}{j^2} \Delta_{13} + \frac{1}{i^2} \Delta_{23} + \frac{2}{j \cdot k} c_1 + \frac{2}{i \cdot k} c_2 + \frac{2}{i\cdot j} c_3 \right): .\notag
\end{gather}
We define $\widetilde{f}_a$ for all $a \in A^1(S) \oplus \BQ\delta$ by linearity in $a$.
If $(a,a) \neq 0$ we also set
\[ f_a = \frac{1}{(a,a)} \widetilde{f}_a. \]

The following implies Theorem~\ref{Thm1}.
\begin{thm} \label{Thm2} Let $S$ be a smooth projective K3 surface and let $n \geq 1$ be an integer.
\begin{enumerate}
\item[(a)] For every $a \in A^1(S) \oplus \BQ \delta$ we have
\[ [h, e_a] = 2 e_a, \quad [h, \widetilde{f}_a] = -2 \widetilde{f}_a, \quad [e_a, \widetilde{f}_a] = (a,a) h \]
as operators on $A^{\ast}(\Hilb^n(S))$.
\item[(b)] If $(a,a) \neq 0$ then $(e_a, f_a, h)$ specializes to a Lefschetz triple in cohomology.
\item[(c)] The Lie subalgebra of $A^{\ast}(\Hilb^n(S) \times \Hilb^n(S))$ generated by $e_a, \widetilde{f}_a, h$ for all $a \in A^1(S) \oplus \BQ \delta$
is isomorphic to $\so( A^1(\Hilb^n(S)) \oplus U_{\BQ} )$.
\end{enumerate}
\end{thm}

We make several remarks.

\begin{rmk}
Consider Lefschetz duals on Hilbert schemes of points of arbitrary smooth projective surfaces $S$.
By the discussion in Section~\ref{section_surface_part} below,
for any $\alpha \in H^2(S,\BQ)$ of non-zero square
the operator $e_{\alpha}$ 
admits the Lefschetz dual $\widetilde{f}_{\alpha}/(\alpha \cdot \alpha)$
where $\widetilde{f}_\alpha$ is defined as in \eqref{ft_a}.
However, the Lefschetz dual of more general elements $a \in H^2(\Hilb^n(S))$
do not seem to admit a nice expression in terms of Nakajima operators.
For example on $\Hilb^n(\p^2)$ we have in general $[f_a, f_b] \neq 0$ and
computer calculations suggest that the expression for $f_a$ involves expressions in Nakajima operators $\Fq_n$ of arbitrarily high degree.
The fact that the Lefschetz duals on $\Hilb (\textup{K3})$ can be expressed as quadratic and cubics in Nakajima operators is remarkable. 
It requires both $K=0$ and $e(S) = 24$, see the proof below.
\end{rmk}

\begin{rmk} \label{Remark_h}
Let $X = \Hilb^n(S)$ where $S$ is a K3 surface.
We consider the action of $h$ on the Chow group of zero-dimensional cycles.
For example let $x_1 ,\ldots, x_n \in S$ be distinct points and consider the subscheme
$z = \{ x_1 , \ldots, x_n \} \in X$. 
Then
\[ h([z]) = 2\sum_{j=1}^{n} [ \{ x_1 , \ldots, x_{j-1}, x_{j+1}, \ldots, x_n, c_0 \} ] \]
where $c_0 \in S$ is a representative of the Beauville-Voisin class $c \in A^2(S)$.
More generally, using the commutation relations one checks that
the action of $h$ on $A_0(X)$ is diagonalizable with eigenvalues $0,2, \ldots, 2n$ and corresponding eigenspaces
\[ A_0(X)_{2i} = \mathrm{Span}_{\BQ}\left( \Fq_1([p_1] - c) \cdots \Fq_1([p_i]-c) \Fq_1(c)^{n-i} 1\, \middle|\, p_1, \ldots, p_i \in S \right). \]
The eigenspace decomposition recovers the proposed splitting of the conjectural Bloch-Beilinson filtration given in \cite{Vial}.
\end{rmk}

\subsection{The surface part} \label{section_surface_part}
We begin with some general remarks that hold for every smooth projective surface $S$.
For a correspondence $\Gamma \in A^{\ast}(S \times S)$ we let $\Gamma'$ be its transpose
which is defined as $\tau_{\ast}(\Gamma)$ where $\tau$ is the automorphism of $S^2$ that swaps the factors.
The correspondence $\Gamma$ acts on $A^{\ast}(S)$ via
\[ \Gamma(\gamma) = \pi_{2 \ast}( \pi_1^{\ast}(\gamma) \cdot \Gamma ). \]
Given two correspondences $\Gamma$ and $\widetilde{\Gamma}$ their composition as operators on $A^{\ast}(S)$ is
\[ \Gamma \circ \widetilde{\Gamma} = \pi_{13 \ast}( \widetilde{\Gamma}_{12} \cdot \Gamma_{23} ), \]
where $\pi_{13}:S^3 \to S \times S$ is the projection to the outer factors.

Let $\deg(\Gamma)$ denote the degree of the homogeneous correspondence $\Gamma$, that is $\Gamma \in A^{\deg(\Gamma)}(S \times S)$.
Define the following operator on $A^{\ast}(\Hilb(S))$:
\begin{equation} T_{\Gamma} = - \sum_{n > 0} n^{\deg(\Gamma) - 3} \Fq_{n} \Fq_{-n}( \Gamma' ). \label{def_T_Gamma} \end{equation}
\begin{lemma} \label{Lemma_TT}
For any $\mathsf{C} \in A^{\ast}(S^k)$ and homogeneous correspondence $\Gamma$,
\begin{multline*}
[ T_{\Gamma}, \Fq_{n_1} \cdots \Fq_{n_k} (C) ]
= \sum_{i : n_i >0} n_i^{\deg(\Gamma)-2} \Fq_{n_1} \cdots \Fq_{n_k}( \id_{S^{i-1}} \times \Gamma \times \id_{S^{k-i}} (\mathsf{C})) \\
+  (-1)^{\deg(\Gamma)-3} \sum_{i : n_i <0} n_i^{\deg(\Gamma)-2} \Fq_{n_1} \cdots \Fq_{n_k}( \id_{S^{i-1}} \times \Gamma' \times \id_{S^{k-i}} (\mathsf{C})).
\end{multline*}
\end{lemma}
\begin{proof}
We commute $T_{\Gamma}$ through the Nakajima operators. If $n_i > 0$ then the $i$-th term contributes
\begin{align*}
& -n_i^{\deg(\Gamma)-3} \Fq_{n_1} \cdots \underbrace{\Fq_{n_i} [ \Fq_{-n_i}}_{\text{from } T_{\Gamma}}, \Fq_{n_i} ] \Fq_{n_{i+1}} \cdots \Fq_{n_k} ( \mathsf{C}_{\{i,i+1 \}^{c}} \cdot \Gamma'_{i,i+1} ) \\
= & n_i^{\deg(\Gamma)-2} \Fq_{n_1} \cdots \Fq_{n_k}( \pi_{\{i+1, i+2\}^c \ast}( \Delta_{i+1, i+2} \cdot \mathsf{C}_{\{i,i+1\}^{c}} \cdot \Gamma'_{i,i+1} )) \\
= & n_i^{\deg(\Gamma)-2} \Fq_{n_1} \cdots \Fq_{n_k}( \id_{S^{i-1}} \times \Gamma \times \id_{S^{k-i}} (\mathsf{C}))
\end{align*}
where we write $\mathsf{C}_{\{i,i+1\}^{c}}$ for the pullback of $\mathsf{C}$ to $S^{k+2}$ along the projection
which forgets the factors $i$ and $i+1$, etc.
The case $n_i < 0$ is similar. 
\end{proof}
\begin{cor} \label{cor_T} $[ T_{\Gamma}, T_{\widetilde{\Gamma}} ] = T_{[\Gamma, \widetilde{\Gamma}]}$ for any homogeneous correspondences $\Gamma, \widetilde{\Gamma}$. \end{cor}
By Corollary~\ref{cor_T} for every $n \geq 1$ we have an embedding of Lie algebras
\[ T: A^{\ast}(S \times S) \to A^{\ast}( \Hilb^n(S) \times \Hilb^n(S) ), \, \Gamma \mapsto T_{\Gamma}. \]

We now specialize to the case of K3 surfaces.
For every $\alpha \in A^1(S)$ consider the correspondences
\begin{equation} \label{35wefs} 
e_{\alpha} = \Delta_{\ast}(\alpha) = c_1 \alpha_2 + \alpha_1 c_2, \quad
\widetilde{f}_{\alpha} = 2(\alpha_1 + \alpha_2), \quad 
h = 2 (c_2 - c_1).
\end{equation}
Either by a direct check or because $\BQ 1 \oplus A^1(S) \oplus \BQ c$ is an invariant subring of $A^{\ast}(S)$ which injects into cohomology,
the correspondences \eqref{35wefs} satisfy the relations of part (a) of Theorem~\ref{Thm2}.
Applying $T$ to these correpondences precisely yields the operators \eqref{e_a}, \eqref{def_h}, \eqref{ft_a}.
Using Corollary~\ref{cor_T} we conclude that Theorem~\ref{Thm2}(a) holds for all $n \geq 1$ and $\alpha \in A^1(S)$.

Further, by the cohomological version of Lemma~\ref{Lemma_TT}
for every $u \in H^i(S)$ and for all $n \in \BZ$ one has 
\[ [h, \Fq_n(u)] = (i-2) \Fq_n(u). \]
Since the Nakajima operators generate the cohomology of Hilbert schemes and $\Fq_n(u)$ is of degree $i$,
we conclude that $h$ acts on $H^j(\Hilb^n(S))$ by multiplication by $j-2n$. This shows part (b) of Theorem~\ref{Thm2} for all $\alpha \in A^1(S)$.
On $S$ the Lie algebra generated by the correspondences \eqref{35wefs} for all $\alpha \in A^1(S)$ is $\so( A^1(S) \oplus U )$ (e.g.
use again the argument with the invariant subring and
that in cohomology we know the result from Verbitsky).
Applying $T$ proves the same on $\Hilb^n(S)$.

\subsection{The general case}
For all $a,b \in A^1(S) \oplus \BZ \delta$ let $\kappa_{ab} = [e_a, \widetilde{f}_b]$.
To prove the remainder of Theorem~\ref{Thm2} we need to establish the following commutation relations:
\begin{gather*}
[h, e_a] = 2 e_a, \quad [h, \widetilde{f}_a] = -2 \widetilde{f}_a, \quad [h, \kappa_{ab}] = 0 \\
[e_a, e_b] = 0, \quad [\widetilde{f}_a, \widetilde{f}_b] = 0, \quad [e_a, \widetilde{f}_a] = (a,a) h \\
\kappa_{ab} + \kappa_{ba} = 2 (a,b) h, \\
[\kappa_{ab}, e_c] = 2 (a,b) e_c + 2(b,c) e_a - 2 (a,c) e_b \\
[\kappa_{ab}, \widetilde{f}_c] = -2 (a,b) \widetilde{f}_c + 2(b,c) \widetilde{f}_a - 2(a,c) \widetilde{f}_b \\
\frac{1}{2} [\kappa_{ab}, \kappa_{cd}] =  (a,d) \kappa_{bc} -(a,c) \kappa_{bd} - (b,d) \kappa_{ac} + (b,c) \kappa_{ad} + ((a,c) (b,d)-(a,d)(b,c)) h.
\end{gather*}
By the discussion in Section~\ref{section_surface_part} we know these relations when all classes involved are from $A^1(S)$.
Moreover it suffices to check the relations on a basis, and 
we only need to check those relations that do not follow from the Jacobi identity and previously established relations.
Hence it is enough to check for all $\alpha, \beta \in A^1(S)$ the following.
\begin{enumerate}
\item[(a)] $[h, e_{\delta}] = 2 e_{\delta}$ and $[h, \widetilde{f}_{\delta}] = -2 \widetilde{f}_{\delta}$
\item[(b)] $[\widetilde{f}_{\alpha}, \widetilde{f}_{\delta}] = 0$.
\item[(c)] $[e_{\delta}, \widetilde{f}_{\delta}] = (2-2n) h$ on $A^{\ast}(\Hilb^n(S))$.
\item[(d)] $\kappa_{\alpha \delta} = -\kappa_{\delta \alpha}$
\item[(e)] $[h, \kappa_{\alpha \delta}] = 0$
\item[(f)] $[\kappa_{\alpha \beta}, e_{\delta}] = 2 (\alpha, \beta) e_{\delta}$.
\item[(g)] $[\kappa_{\alpha \beta}, \widetilde{f}_{\delta}] = -2 (\alpha, \beta) \widetilde{f}_{\delta}$.
\end{enumerate}
We check part (c) below in detail by a direct computation.
The remaining relations follow from a straightforward application of Lemma~\ref{Lemma_TT} and we skip the details.
One may see them also as follows: Each is a relation between Nakajima operators of degree at most $5$,
which after applying the commutation relations \eqref{Nak_com}, reduces to a relation
in $S^k$ for some $k \leq 5$ between classes which are polynomials in $\Delta_{ij}$, $c_i$ and $\alpha_j$ for some $\alpha \in A^1(S)$.
By the result of Verbitsky \cite{V} we know these relations hold in cohomology.
Since $k \leq 5$ we hence know from work of Voisin \cite{Voisin} (see also \cite{Yin}) that they hold in Chow as well.

\begin{proof}[Proof of Relation (c)]
By \cite[Thm.1.6]{MN} the operator
\[ L_0 = \sum_{k>0} \Fq_{k} \Fq_{-k}(\Delta) \]
acts by multiplication by $-n$ on $A^{\ast}(\Hilb^n(S))$. Hence we need to show
\[ [e_{\delta}, \widetilde{f}_{\delta}] = 2h + 2 L_0(1) h. \]
We do this by expanding both sides in Nakajima operators ordered in normal product ordering.
For the right hand side we obtain
\[ 2h + 2 L_0(1) h
= 4 \sum_{k>0} \frac{1-k}{k} \Fq_k \Fq_{-k}(c_2 - c_1) + 4 \sum_{k, \ell > 0} \frac{1}{k} :\Fq_k \Fq_{-k}(c_2-c_1) \Fq_{\ell} \Fq_{-\ell}(\Delta):\]

For the left hand side we first consider the quartic terms, that is those of degree $4$ in a normal ordering.
These involve precisely one interaction of the Nakajima operators.
Let $\CE$ be the argument of $\Fq_i \Fq_j \Fq_k$ in the definition of $\widetilde{f}_{\delta}$.
Since the argument of the cubic term in $e_{\delta}$ and $\widetilde{f}_{\delta}$ is $S_3$-symmetric,
the quartic term reads
\begin{align*}
& 9 \cdot (-\frac{1}{6}) \cdot (-\frac{1}{3}) \sum_{\substack{j_1 + k_1 = -i_1 \\ j_2 + k_2 = i_1}}
:[\Fq_{i_1}, \Fq_{-i_1}] \Fq_{j_1} \Fq_{k_1} \Fq_{j_2} \Fq_{k_2} ( \Delta_{134} \cdot \CE_{256} ): \\
& = \frac{1}{2} \sum_{\substack{j_1 + k_1 = -i_1 \\ j_2 + k_2 = i_1}}
i_1 :\Fq_{j_1} \Fq_{k_1} \Fq_{j_2} \Fq_{k_2} ( \pi_{3456 \ast}( \Delta_{134} \cdot \CE_{256} \cdot \Delta_{12} )): \\
& = 
\frac{1}{2} \sum_{\substack{a+b+c+d=0 \\ a+b \neq 0}}
:\Fq_{a} \Fq_b \Fq_c \Fq_d\left( 2 \frac{c+d}{d^2} \Delta_{123} - \frac{4}{d} \Delta_{12} c_3 + 2 \frac{(c+d)}{c \cdot d} c_1 c_2 + \frac{1}{(c+d)} \Delta_{12} \Delta_{34} \right):.
\end{align*}
The term with $\Delta_{12} \Delta_{34}$ cancels by symmetrizing.
For the remaining terms we insert the decomposition \eqref{decomp_delta123} of the small diagonal $\Delta_{123}$
and observe that the sum vanishes when it is taken over all $a,b,c,d$ such that $a+b+c+d = 0$.
The terms we overcounted (those with $a+b=0$) sum up precisely to (the negative of) the quartic term in $2h+2 L_0(1) h$.

For the quadratic term we have two Nakajima interactions. We get
\begin{multline*}
\quad \text{Quadratic terms in } [e_{\delta}, \widetilde{f}_{\delta}]
= 
\sum_{\substack{ i_1 + j_1 + k_1=0 \\ i_1, j_1 < 0}} [ \Fq_{i_1}, \Fq_{-i_1}] [\Fq_{j_1}, \Fq_{-j_1}] \Fq_{k_1} \Fq_{-k_1} ( \Delta_{135} \cdot \CE_{246} ) \\
+ \sum_{\substack{ i_1 + j_1 + k_1=0 \\ i_1, j_1 > 0}} [ \Fq_{i_1}, \Fq_{-i_1}] [\Fq_{-j_1}, \Fq_{j_1}] \Fq_{-k_1} \Fq_{k_1} ( \Delta_{146} \cdot \CE_{235} ).
\end{multline*}
Using the Nakajima commutation relations and $\Delta \cdot \Delta = e(S) \cdot c_1 c_2$ this simplifies as desired to
\begin{align*} 
\sum_{k > 0} \Big( 4 (k-1) - \sum_{\substack{i+j=k\\i,j>0}} e(S) \frac{i \cdot j}{k^2} \Big) \Fq_{k} \Fq_{-k}(c_2 - c_1)
& = 4 \sum_{k > 0} \frac{1-k}{k} \Fq_k \Fq_{-k}(c_2 - c_1),
\end{align*}
where in the last equality we have used $\sum_{i+j=k} i \cdot j = \frac{1}{6} k (k^2-1)$ and $e(S) = 24$.
\end{proof}


\begin{thebibliography}{10}
\bibitem{Ar}
D.~Arapura,
{\em Motivation for Hodge cycles},
Adv. Math. {\bf 207} (2006), no. 2, 762--781. 

\bibitem{BV}
A.~Beauville, C.~Voisin,
{\em On the Chow ring of a K3 surface},
J. Algebraic Geom. {\bf 13} (2004), no. 3, 417--426. 

\bibitem{CM}
F.~Charles, E.~Markman,
{\em The standard conjectures for holomorphic symplectic varieties deformation equivalent to Hilbert schemes of K3 surfaces},
Compos. Math. {\bf 149} (2013), no. 3, 481--494. 

\bibitem{Groj}
I.~Grojnowski, \emph{ Instantons and affine algebras. {I}. {T}he {H}ilbert
  scheme and vertex operators},
\newblock Math. Res. Lett. \textbf{ 3} (1996), no. 2, 275--291.

\bibitem{Kl}
S.~L.~Kleiman,
{\em Algebraic cycles and the Weil conjectures}. Dix expos\'es sur la cohomologie des schémas, 359--386, 
Adv. Stud. Pure Math., 3, North-Holland, Amsterdam, 1968. 

\bibitem{Lehn}
M.~Lehn, \emph{ Chern classes of tautological sheaves on {H}ilbert schemes of
  points on surfaces},
\newblock Invent. Math. \textbf{ 136} (1999), no. 1, 157--207.

\bibitem{LL}
E.~Looijenga, V.~A.~Lunts
{\em A Lie algebra attached to a projective variety},
Invent. Math. {\bf 129} (1997), no. 2, 361--412. 

\bibitem{Markman}
E.~Markman,
{\em On the monodromy of moduli spaces of sheaves on K3 surfaces},
J. Algebraic Geom. {\bf 17} (2008), no. 1, 29--99. 

\bibitem{Markman2}
E.~Markman,
{\em A survey of Torelli and monodromy results for holomorphic-symplectic varieties},
Complex and differential geometry, 257--322,
Springer Proc. Math., {\bf 8}, Springer, Heidelberg, 2011. 

\bibitem{MN} D.~Maulik, A.~Negut,
\emph{Lehn's formula in Chow and Conjectures of Beauville and Voisin},
\href{https://arxiv.org/abs/1904.05262}{arXiv:1904.05262}

\bibitem{Mon} B.~Moonen,
{\em On the Chow motive of an abelian scheme with non-trivial endomorphisms},
J. Reine Angew. Math. {\bf 711} (2016), 75--109.

\bibitem{Nak}
H.~Nakajima, \emph{ Heisenberg algebra and {H}ilbert schemes of points on
  projective surfaces},
\newblock Ann. of Math. (2) \textbf{ 145} (1997), no. 2, 379--388.

\bibitem{NOY}
A.~Negut, G.~Oberdieck, Q.~Yin,
{\em Motivic decompositions for the Hilbert scheme of points of a K3 surface},
\href{https://arxiv.org/abs/1912.09320}{arXiv:1912.09320}

\bibitem{V}
M.~Verbitsky,
{\em Cohomology of compact hyper-Kähler manifolds and its applications},
Geom. Funct. Anal. {\bf 6} (1996), no. 4, 601--611. 

\bibitem{Vial}
C.~Vial,
{\em On the motive of some hyperK\"ahler varieties},
J. Reine Angew. Math. {\bf 725} (2017), 235--247. 

\bibitem{Voisin}
C.~Voisin,
{\em On the Chow ring of certain algebraic hyper-K\"ahler manifolds},
Pure Appl. Math. Q. {\bf 4} (2008), no. 3, Special Issue: In honor of Fedor Bogomolov. Part 2, 613--649. 

\bibitem{Yin}
Q.~Yin,
{\em Finite-dimensionality and cycles on powers of K3 surfaces},
Comment. Math. Helv. {\bf 90} (2015), no. 2, 503--511. 

\end{thebibliography}
\end{document}